\documentclass[12pt]{amsart}
\usepackage{amsmath,amssymb}
\usepackage[all,cmtip]{xy}
\usepackage{hyperref}

\newcommand{\bC}{\mathbb C}

\newcommand{\bP}{\mathbb P}
\newcommand{\bQ}{\mathbb Q}
\newcommand{\bZ}{\mathbb Z}

\newcommand{\cC}{\mathcal C}
\newcommand{\cE}{\mathcal{E}}
\newcommand{\cF}{\mathcal F}

\newcommand{\cG}{\mathcal G}
\newcommand{\cH}{\mathcal H}
\newcommand{\cI}{\mathcal I}

\newcommand{\cO}{\mathcal O}

\newcommand{\cQ}{\mathcal Q}
\newcommand{\cR}{\mathcal R}

\newcommand{\cU}{\mathcal U}

\newcommand{\ra}{\rightarrow}

\newcommand{\rH}{\mathrm{H}}

\newcommand{\wH}{\widetilde{H}}
\newcommand{\wP}{\widetilde{P}}
\newcommand{\wT}{\widetilde{T}}
\newcommand{\wV}{\widetilde{V}}
\newcommand{\wW}{\widetilde{W}}
\newcommand{\wXi}{\widetilde{\Xi}}

\newcommand{\wcW}{\widetilde{\mathcal{W}}}

\newcommand{\Aut}{\operatorname{Aut}}
\newcommand{\Bl}{\operatorname{Bl}}
\newcommand{\Br}{\operatorname{Br}}
\newcommand{\Coh}{\operatorname{Coh}}

\newcommand{\Ext}{\operatorname{Ext}}

\newcommand{\Gr}{\operatorname{Gr}}
\newcommand{\Hilb}{\operatorname{Hilb}}
\newcommand{\Hom}{\operatorname{Hom}}
\newcommand{\Moduli}{\operatorname{Moduli}}

\newcommand{\Pic}{\operatorname{Pic}}

\newcommand{\PGL}{\operatorname{PGL}}
\newcommand{\Sec}{\operatorname{Sec}}

\newcommand{\spa}{\operatorname{span}}
\newcommand{\Sing}{\operatorname{Sing}}

\newcommand{\Sym}{\operatorname{Sym}}

\theoremstyle{plain}
\newtheorem{prop}{Proposition}[section]

\newtheorem{theo}[prop]{Theorem}
\newtheorem{coro}[prop]{Corollary}
\newtheorem{lemm}[prop]{Lemma}

\theoremstyle{definition}

\newtheorem{ques}[prop]{Question}

\newtheorem{rema}[prop]{Remark}
\newtheorem{exam}[prop]{Example}

\title{Cubic fourfolds of discriminant $24$ and rationality}
\author{Brendan Hassett}
\address{Department of Mathematics\\
Brown University
Box 1917 \\
151 Thayer Street,
Providence, RI 02912 \\
USA}
\email{brendan\underline{ }hassett@brown.edu}

\begin{document}
\maketitle

Let $X$ be a smooth complex cubic fourfold; when is it rational?
Classically, this was studied for specific examples obtained via uniform 
geometric constructions \cite{Fano,BD,Tregub}. 
Today, we consider the moduli space
$\cC$ of smooth cubic fourfolds.
The locus $\cC_{\mathrm{rat}}\subset \cC$ of rational members is a countable union of Zariski-closed subsets $\cC_{\varrho}\subset \cC$ \cite{KT},
which we seek to enumerate.  

All known examples of rational parametrizations $\bP^4 \stackrel{\sim}{\dashrightarrow} X$
blow up a K3 surface.  Thus we characterize cubic fourfolds whose middle cohomology -- as a Hodge structure --
contains the primitive cohomology of some K3 surface: This is an explicit countable
union of divisors in $\cC$ \cite{Has}. (We follows its notation.) Kuznetsov conjectured that this locus
is precisely $\cC_{\mathrm{rat}}$ \cite{Kuz,AT,Add,HasSurv}; in particular, the very general
cubic fourfold is irrational. Finding irrational cubic fourfolds remains a challenging
problem, despite promising ideas of Katzarkov, Kontsevich, Pantev, and Yu.

We focus on proving rationality for cubic fourfolds associated
with K3 surfaces. Russo and Staglian\`o \cite{RS1,RS2} found new divisors 
in $\cC_{26},\cC_{38},\cC_{42} \subset \cC$ parametrizing rational $X$. 
Numerous examples
have also been found in codimension two \cite{HasJAG,AHTV,FL}. 
A common theme is assigning
{\em twisted} K3 surfaces to cubic fourfolds, i.e., pairs $(S,\alpha)$ where
$S$ is a K3 surface and $\alpha \in \Br(S)$ a Brauer class, such that we have saturated embeddings
of Hodge structures
$$T(S,\alpha) \hookrightarrow H^4(X,\bZ)(1).$$
Here $T(S) \subset H^2(S,\bZ)$ is transcendental cohomology and $T(S,\alpha) \subset T(S)$
is a finite index sublattice. The classification of cubic fourfolds with associated twisted K3 surfaces
is in \cite{HuyTwist}; $X$ should be rational whenever $\alpha=0$.
In \cite{HasJAG} and \cite{AHTV}, $S$ has degree two and $\alpha$ has order two and three,
respectively. These form divisors $\cC_8,\cC_{18}\subset \cC$, each containing countably many
 codimension-two loci in $\cC_{\mathrm{rat}}$.

We offer a new example in this vein: Consider cubic fourfolds $X$ with twisted K3 surfaces $(S,\alpha)$,
where $S$ has degree six and $\alpha$ has order two,
forming a divisor in $\cC_{24}\subset \cC$.
These are rational whenever $\alpha=0$, a countable union of codimension-two loci; see Theorems~\ref{theo:makerat}
and \ref{theo:critrat}.

\

Section~\ref{sect:RC} presents our rationality construction via {\em nodal} sextic del Pezzo surfaces in $X$.
This differs substantially from \cite{AHTV}, where we considered smooth sextic del Pezzo surfaces.  
Section~\ref{sect:HTHG} puts this geometry in Hodge-theoretic terms, allowing an explicit lattice
theoretic description of where the construction applies.  The condition is expressed in terms of the
hyperk\"ahler geometry of the variety of lines $F_1(X)$. Section~\ref{sect:syn} is the bridge
between the hyperk\"ahler geometry -- which can be understood in uniform terms -- and the 
del Pezzo fibrations. Section~\ref{sect:BoS} gives context for a key construction needed in
Section~\ref{sect:syn}, formulated in Theorem~\ref{theo:uniquescroll}.

Ultimately, to automate rationality constructions we want to 
sidestep clever classical constructions, replacing them with moduli-theoretic analysis.  
We return to our example from the perspective of twisted sheaves
on K3 surfaces in Section~\ref{sect:TSI}; the main result is Theorem~\ref{theo:gettwist}.
The eight-dimensional hyperk\"ahler manifolds arising also have intriguing properties;
see Proposition~\ref{prop:P4}.  

Here is the high-level strategy for the rationality proof in the complex case: 
Let $X$ be a smooth complex cubic fourfold with variety of lines $F_1(X)$. 
Suppose that $X$ is special of discriminant $24$, i.e. it has Hodge class $W\in H^4(X,\bZ)$ 
with degree six and self-intersection $20$ (see (\ref{eqn:disc24})). The variety of lines contains divisors  
$E,E' \subset F_1(X)$, both ruled over a degree-six K3 surface $S$; 
each ruling yields a scroll $T \subset X$ (\S \ref{subsect:HGVoL}). The scroll is
degree six and generically has two nodes (\S \ref{subsect:GotS}). A cubic fourfold containing
such a scroll contains a nodal sextic del Pezzo surface $W$ and {\em vice versa} (\S \ref{subsect:relSdP});
the constructions use three-dimensional scrolls.  The linear system of quadrics cutting out $W$ gives a 
fibration over $\bP^2$ (\S \ref{subsect:ELS}) whose fibers are also nodal sextic del Pezzo surface
(\S \ref{subsect:RC}).  Rationality follows provided the fibration has a multisection of odd degree (\S \ref{subsect:RP}),
because sextic del Pezzo surfaces with zero-cycles of degree one are rational.

\subsection*{Acknowledgments} 
The author was partially supported by Simons Foundation Award 546235 and NSF grants 1701659 and 1929284.  
He is grateful to Yuri Tschinkel for helpful conversations about this project.

\section{Rationality construction}
\label{sect:RC}
We work an algebraically-closed base field $k$. 

\subsection{Sextic del Pezzo surfaces:} \label{subsect:SdPS}
Let $\wW \subset \bP^6$ denote a sextic del Pezzo surface,
isomorphic to $\bP^2$ blown up in three non-collinear points $p_1,p_2,p_3$, and embedded anti-canonically. 
It contains a distinguished
hexagon of lines
$$ E_1, \ell_{12}, E_2, \ell_{23}, E_3, \ell_{13},$$
where $E_i$ is the exceptional divisor over $p_i$ and $\ell_{ij}$ is the proper transform
of the line joining $p_i$ and $p_j$.  We have intersection numbers
\begin{align*}
E_i^2=\ell_{ij}^2=-1,&\quad i,j \in \{1,2,3\} \\
E_1\ell_{12}=\ell_{12}E_2=E_2\ell_{23}&=\ell_{23}E_3=E_3\ell_{13}=\ell_{13}E_1=1
\end{align*}
with the remaining intersections zero. The complement to this hexagon in $\wW$ is a two-dimensional torus 
$U \simeq k^*\times k^*$.

The three {\em conic fibrations} arise from the linear series 
\begin{equation} \label{eqn:3conics}
| C_i |:\wW \stackrel{\gamma_i}{\longrightarrow} \bP^1, \quad C_i = \ell_{ij}+E_j = \ell_{ik}+E_k.
\end{equation}
The two {\em blowup realizations}  come from the series
\begin{align*}
| R |:\wW \stackrel{b}{\longrightarrow} \bP^2 , \quad &R=\ell_{12}+E_1+E_2=\ell_{13}+E_1+E_3=\ell_{23}+E_2+E_3, \\
| R' |:\wW \stackrel{b'}{\longrightarrow} \bP^2 , \quad &R' = E_1+\ell_{12}+\ell_{13} = E_2 +\ell_{12}+\ell_{23} + E_3 + \ell_{13}+\ell_{23}.
\end{align*}

\begin{rema} \label{rema:blowdown}
Let $\wW$ be a sextic del Pezzo surface over an arbitrary field $K$.  
Each conic fibration on $\wW$ yields a
\'etale algebra of dimension three over $K$. 
Given a rational point 
$$s \in \wW(K)\subset \bP^6$$
away from the hexagon of lines, double projection from $s$ gives a birational 
map to a quadric surface
\begin{equation} \label{eqn:PiBD}
\Pi(s): \Bl_s(\wW) \stackrel{\sim}{\longrightarrow} QS \subset \bP^3,
\end{equation}
blowing down the three conics 
$$s \in C_1(s),C_2(s),C_3(s) \subset \wW, \quad C_i(s)=\gamma_i^{-1}(\gamma_i(s)).$$
\end{rema}

\subsection{Nodal sextic del Pezzo surfaces} \label{subsect:NDPS}

Choose distinct points $w_+,w_- \in U$ that are generic in the sense that $\gamma_i(w_+)\neq \gamma_i(w_-)$ for $i=1,2,3$.  
Project from a point on the secant line joining these points
$$ x \in \Sec(w_+,w_-) \setminus \{w_0,w_-\} $$
to obtain:
\begin{align*}
\pi_x:\bP^6 & \dashrightarrow  \bP^5 \\
\wW  & \longrightarrow W:=\pi_x(\wW) \\
\{w_+,w_-\} & \mapsto w_0 := \pi_x(\mathrm{Sec}(w_+,w_-)).
\end{align*}
The resulting surface $W \subset \bP^5$ will be called a {\em nodal sextic del Pezzo surface}.

\begin{prop} \label{prop:Wbasic}
A nodal sextic del Pezzo surface $W$ has the following properties:
\begin{itemize}
\item
$\wW$ is the normalization of $W$, whose only singularities is a node (transverse self-intersection) at $w_0$.
\item
$W$ contains nodal plane curve curves $N,N'$, singular at $w_0$, whose normalizations
$\widetilde{N}, {\widetilde{N}}'$ are the unique members of $|B|$ and $|B'|$ containing $\{w_+,w_-\}$.
\item
The spanning planes $P\supset N$ and $P'\supset N'$ are contained in all the quadrics
vanishing along $W$.
\item
$W$ depends on two moduli as an abstract variety and three moduli as a polarized variety.
\end{itemize}
\end{prop}
\begin{proof}
The ideal of $\wW$ is generated in degree two, so secant lines not contained in $\wW$ meet it
in exactly two points.  
We assumed that $w_+$ and $w_-$ are not on the six lines, nor are both on the same conic.
It follows that they cannot be extended to a set of four coplanar points on $\wW$ \cite[Prop.~1]{EGH},
so $x$ lies on a unique secant to $\wW$. This implies that $W$ has the indicated singularities.

Our genericity assumption for $\{w_+,w_-\}$ also ensures that $\widetilde{N}$ and ${\widetilde{N}}'$ are smooth twisted cubics on $\wW$,
thus they map to nodal cubic plane curves under projection. 
Since $P\cap W$ and $P'\cap W$ are plane cubic curves, these planes must be contained in the linear system of quadrics 
vanishing on $W$.  

For the moduli statement, we observe that the choice of $\{w_+,w_-\}$, modulo the torus action, depends
on two parameters. The choice of $x \in \Sec(w_+,w_-)$ -- equivalent to stipulating the embedding line bundle
on $W$ -- involves one additional parameter encoding an identification 
$$\cO_{\wW}(1)|_{w_+}\simeq \cO_{\wW}(1)|_{w_-}.$$
\end{proof}

\subsection{Equations and linear series} \label{subsect:ELS}
We continue to assume
$W \subset \bP^5$ is a nodal sextic del Pezzo surface. 

Recall that a coherent sheaf $\cF$ on $\bP^r$ is {\em $d$-regular},
in the sense of Castelnuovo and Mumford, if $H^i(\cF(d-i))=0$ for all $i\ge 1$. 
See \cite[ch.~4]{EisSyz} for the properties of regularity, for instance:
\begin{itemize}
\item{If $\cF$ is $d$-regular then $\cF(d)$ is globally generated \cite[Cor.~4.18]{EisSyz}.}
\item{
Let $H \subset \bP^r$ be a hyperplane that is not a zero divisor of $X$;
then $\cI_X$ is $d$-regular iff $\cI_{X\cap H}$ is $d$-regular and $H^1(\cI_X(d-1))=0$
\cite[Ex.~4.14]{EisSyz}.
The last condition is equivalent to the surjectivity of
$\Gamma(\cO_{\bP^r}(d-1)) \rightarrow \Gamma(\cO_X(d-1))$.}
\end{itemize}

\begin{prop} \label{prop:equations}
The ideal sheaf $\cI_W$ is $3$-regular thus generated in degree $\le 3$
\begin{equation}
\cI_W = \left<Q_0,Q_1,Q_2,F,F'\right>,  \label{eqn:ideal}
\end{equation}
where the $Q_i$ are quadratic forms and $F$ and $F'$ are cubic forms. 
\end{prop}
\begin{rema} \label{rema:K3}
The surface 
$$S=\{Q_0=Q_1=Q_2=0\} = P \cup_N W \cup_{N'} P'$$ 
is a degenerate K3 surface of degree eight. The forms $F$ and $F'$ may be
chosen to cut out $N\subset P$ and $N' \subset P$ respectively.  
\end{rema}
\begin{proof}
A generic hyperplane section $C \subset X$ is a smooth
curve of genus one and degree six, thus $\cI_C \subset \cO_{\bP^4}$ is $3$-regular \cite[\S 8A]{EisSyz}.
It remains then to analyze $H^1(\cI_W(2))$; suppose $H^1(\cI_W(2))\neq 0$.
We deduce that $h^0(\cI_W(2))>3$ and also $h^0(\cI_C(2))>3$, as $W\subset \bP^5$ is nondegenerate; 
but then $h^1(\cI_C(2))\neq 0$, violating the $3$-regularity of $\cI_C$. 

We therefore have
$$\Gamma(\cO_{\bP^5}(3)) \twoheadrightarrow \Gamma(\cO_W(3)) \simeq k^{36},$$
hence $h^0(\cI_W(3))=20$, giving (\ref{eqn:ideal}). 
\end{proof}

Propositions~\ref{prop:Wbasic} and \ref{prop:equations} imply:
\begin{coro} \label{coro:dim22}
The moduli space of incidences
$$\mathcal{V} = \{(W,X): \text{nodal sextic del Pezzo surface }  W\subset X \text{ cubic fourfold} \}$$
is unirational of dimension $3+19=22$.  
\end{coro}

\begin{prop} \label{prop:cubicsmooth}
Assume $k$ has characteristic zero. Then a generic cubic fourfold $X\supset W$ is smooth.
\end{prop}
\begin{proof}
Consider the blow up $\wP:=\Bl_W(\bP^5)$.  This resolves the base locus of the linear series $\Gamma(\cI_W(3))$, as the ideal
of $W$ is generated in degree three. Consider the resulting basepoint-free linear series on $\wP$.  

We analyze the singularities of $\wP$ in \'etale-local coordinates centered at $w_0\in W \subset \bP^5$:
$$(0,0,0,0,0)  \in \{ u=vy=vz=xy=xz=0\}. $$
The relevant affine patch on $\wP$, where we see smooth hypersurfaces with tangent space $T_wW$, 
has coordinates  and equations
$$(u,v,x,y,z,b,c,d) \quad u=vye, z=yb, x=vc, d=bc,$$
and thus is nonsingular. Overall, $\wP$ has one isolated singularity, locally equivalent to a cone over the
Segre embedding of $\bP^2\times \bP^2$.  The Bertini Theorem implies the generic member
of our linear series is smooth.  
\end{proof}

Proposition~\ref{prop:equations} also yields:
\begin{coro} \label{coro:Wfibration}
Let $W \subset X$ be a nodal sextic del Pezzo surface in a smooth cubic fourfold.
If $X$ does not contain a plane then $\cI_W \subset \cO_X$ is generated in degree two. 
The quadratic equations for $W$ induce
\begin{equation} \label{eqn:phiW}
\phi_W: \Bl_W(X) \rightarrow \bP(\Gamma(\cI_W(2))^{\vee})\simeq \bP^2,
\end{equation}
a smooth fourfold fibered over $\bP^2$.
\end{coro}
\begin{proof}
Actually, we need only assume that $X$ does not contain the planes $P$ and $P'$ described
in Remark~\ref{rema:K3}.  This ensures that $\Gamma(\cI_W(2))$ -- regarded as a linear 
series on $X$ -- cuts out $W$. Blowing up resolves the indeterminacy and
yields the morphism $\phi_W$.
Now $\Bl_W(X)$ is the blow-up of a smooth fourfold along a surface whose only singularities
are two smooth branches meeting transversally; such blowups are smooth.
(We will return to this local geometry in Lemma~\ref{lemm:local}.)
\end{proof}

\subsection{Residuation constructions} \label{subsect:RC}
We continue to assume that $W$ is a nodal sextic del Pezzo surface singular at $w_0$.

We start with an observation on smooth hyperplane sections $C\subset W$.
These have genus one and degree six so $C\subset \bP^4$ is neither linearly 
normal nor cut out by quadrics; $\cI_C$ is not $2$-regular!  It is $3$-regular
so we have, as in (\ref{eqn:ideal}),
$$\cI_C=\left<q_0,q_1,q_2,f,f'\right>.$$
Choose a pencil in $\spa(q_0,q_1,q_2)$ and an element in $\spa(f,f')$, both generic; after
relabelling this could be written
$$D=\{q_0=q_1=f=0\} = C \cup_B C'.$$
A liaison computation -- $D$ is a nodal complete intersection with $\omega_D\simeq \cO_D(2)$ --
yields 
$$\cO_C(B)=\cO_C(2), \quad \cO_{C'}(B)=\cO_{C'}(2),$$
and $C'$ also has genus one and degree six.  This constructive is involutive in the sense
that $C$ is one of the curves $C''$ arising when this process is applied to $C'$.

We return to analyzing sextic del Pezzo surfaces:
\begin{prop} \label{prop:residual}
Assume $k$ has characteristic zero.
Let $W\subset X\subset \bP^5$ be a nodal sextic del Pezzo surface 
embedded in a smooth cubic fourfold that contains no planes.
Given a generic pencil
$$\Lambda \in  \Gamma(\cI_W(2))\simeq k^3$$
with base locus $Y_{\Lambda} \subset \bP^5$, the residual surface
$$Y_{\Lambda} \cap X = W \cup W'$$
is also a nodal sextic del Pezzo surface.  
\end{prop}
\begin{proof}
The most direct approach is to compute the invariants of $W'$. 
Since $W$ is a nodal sextic del Pezzo surface we have
$$\chi(\cO_W(n))=\chi(\cO_{\wW}(n))-1=3n^2+3n.$$
Since $Y_{\Lambda}\cap X$ is a complete intersection 
$$\chi(\cO_{Y_{\Lambda} \cap X}(n))=6n^2-6n+7.$$

The conductor curve $B=W \cap W'$ is a bit complicated. Its partial
normalization $\widehat{B} \subset \wW$ has nodes at $w_+,w_-$;
write $B^{\nu}$ for the full normalization. 
Adjunction
$$\omega_{Y_{\Lambda}\cap X}=\cO_{Y_{\Lambda}\cap X}(1), \quad
\omega_{Y_{\Lambda}\cap X}|_{\wW} = \omega_{\wW}(\widehat{B})$$
implies $[\widehat{B}]=-2K_{\wW}$ and has arithmetic genus seven.
Thus $B^{\nu}$ has genus five and $B$ is obtained from it
by gluing four points -- the pre-images of $w_{\pm}$ --
to a quadruple point. In conclusion, $B$ has arithmetic genus eight,
degree $12$, and Hilbert polynomial
$$\chi(\cO_B(n))=12n-7.$$

Putting everything together
\begin{align*}
\chi(\cO_{W'}(n))&=\chi(\cO_{Y_{\Lambda} \cap X}(n)) - \chi(\cO_W(n)) + \chi(\cO_B(n)) \\
			&= (6n^2 - 6n+7) - (3n^2 + 3n) + (12n - 7) = 3n^2 + 3n.
\end{align*}
Furthermore, since $W'$ is the residual to $W$ in a complete intersection of two quadrics, 
it also has a node (transverse double point) at $w_0$. The partial normalization ${\widehat{B}}' \subset W'$
of $B$ also has two nodes and arithmetic genus seven.  The gluing $B^{\nu} \ra \widehat{B}'$ differs
from the gluing $B^{\nu} \ra \widehat{B}$. 

\begin{lemm} \label{lemm:local}
Consider the blow up $\beta:\Bl_W(X)\ra X$ and the morphism associated with the quadratic equations of $W$ 
(see (\ref{eqn:phiW}) above):
$$\phi_W: \Bl_W(X) \rightarrow \bP(\Gamma(\cI_W(2))^{\vee})\simeq \bP^2.$$
\begin{itemize} 
\item{$\beta^{-1}(w_0)\simeq \bP^1 \times \bP^1$;}
\item{given a smooth local surface $w \in \Sigma \subset X$ meeting each branch of $W$ along a smooth curve,
$\Sigma$ lifts to $\Bl_W(X)$ and meets $\bP^1 \times \bP^1$ transversally at a point;}
\item{the proper transform of $W'$ in $\Bl_W(X)$ equals the normalization of $w_0\in W'$;}
\item{the canonical sheaf 
$$\omega_{\Bl_W(X)}=\beta^*\omega_X(E)=(\beta^*\cO_X(-3))(E)$$
where $E$ is the exceptional divisor of $\beta$.}
\end{itemize}
\end{lemm}
\begin{proof}
This is easy to understand through local analytic coordinates on $X$
$$\{u=v=y=z=0\} = w_0 \in W = \{uy=uz=vy=vz=0\}.$$
The blowup of $W$ has equations 
$$zA-yB=vA-uC=vB-uD=zC-yD=AD-BC,$$
which is nonsingular. The fiber over $w_0$ is 
$$\{AD-BC=0\} \simeq \bP^1 \times \bP^1.$$
Since $\Sigma$ meets $W$ as a Cartier divisor it lifts to $\Bl_W(X)$.  Writing
$$T_w\Sigma = \{a_0u+a_1v = b_0 y + b_1z=0\}$$
we see that its intersection with the fiber over $w_0$ is
$$[A,B,C,D]=[a_1b_1,-a_1b_0,-a_0b_1,a_0b_0].$$
In particular, a pair of such surfaces meeting transversally at the origin is separated in 
the blowup.  
The last statement follows from adjunction and the formula for the canonical class of a blowup.
\end{proof}
The last step of the proof of Proposition~\ref{prop:residual}
requires that $k$ have characteristic zero: The generic fiber $\wW'$ of $\phi_W$ is smooth
by the Bertini Theorem. Let $W'$ denote its image in $X$ and recall that
$$\chi(\cO_{\wW'}(n))=\chi(\cO_W(n))+1 = 3n^2 + 3n+1,$$
as expected for a sextic del Pezzo surface.   

Consider the linear series $\Gamma(X,\cI_W(3))$ and the induced linear series 
$$\Gamma(\Bl_W(X),\beta^*\cO_X(3)(-E))\simeq \Gamma(\omega_{\Bl_W(X)}^{-1}),$$
which is basepoint-free by Proposition~\ref{prop:equations}. Furthermore, Corollary~\ref{coro:Wfibration}
implies
$$\Gamma(\Bl_W(X),\beta^*\cO_X(2)(-E))$$
is also basepoint free; for each nonzero section $s \in \Gamma(X,\cI_W(2))$ we have
$$\Gamma(\Bl_w(X),\beta^*\cO_X(1))  \stackrel{s}{\hookrightarrow} \Gamma(\Bl_W(X),\beta^*\cO_X(3)(-E)).$$
For smooth fibers $\wW'$ of $\phi_W$, adjunction yields
$$\Gamma(\Bl_W(X),\omega_{\Bl_W(X)}^{-1}) \twoheadrightarrow \Gamma(\wW', \omega_{\wW'}^{-1})$$
and for $s$ nonvanishing along $\wW'$ we obtain
$$\Gamma(W',\cO_{W'}(1)) \hookrightarrow \Gamma(\wW',\omega_{\wW'}).$$
Thus $\omega_{\wW'}^{-1}$ is ample hence $\wW'$ is a sextic del Pezzo surface.
\end{proof} 

\begin{rema}
We offer a more conceptual way to understand Proposition~\ref{prop:residual},
which applies provided $k$ does not have characteristic two.

Suppose we are given two intersecting planes in $\bP^5_{t,u,v,x,y,z}$
$$P=\{t=u=v=0\}, P'=\{t=y=z=0\}.$$
Consider
$$\Gamma(\cI_{P\cup P'}(2))= t \Gamma(\cO_{\bP^5}(1)) + \spa(uy,uz,vy,vz)$$
which has dimension ten; all members have tangent space $t=0$ at $p=[0,0,0,1,0,0]$.
Any complete intersection $Y$ of two such quadrics contains two additional planes 
$R$ and $R'$ of the form
$$\{t=a_0u+a_1v = b_0 y + b_1z=0\};$$
$P\cup R \cup P' \cup R$ is a cone over a cycle of four rational curves.
After coordinate change, we write
$$R=\{t=v=y=0\}, \quad R' = \{t=z=u=0\},$$
and the defining quadrics as
$$tL_1 + uy = tM_1+ vz = 0, \quad L_1,M_1 \in k[t,u,v,x,y,z] \text{ linear}.$$
Furthermore, we may assume
$$L_1 = x + a t + L_2, \ M_1 = x + bt +  M_2, \quad a,b \in k, \quad L_2,M_2 \in k[u,v,y,z].$$

The point $p$ is generically an ordinary double point of $Y$. Projecting from $p$
-- which entails eliminating $x$ --
yields a quadric hypersurface
$$Q=\{t((a-b)t + (L_2 - M_2)) = vz-uy\} \subset \bP^4_{t,u,v,y,z},$$ 
which is smooth provided $Y$ is chosen generically. 
The inverse of projection 
$$\pi_p:Y \stackrel{\sim}{\dashrightarrow} Q$$
blows up the locus
$$Z=\{t = vz=uy =0 \}.$$
We see that $Y$ has four additional ordinary double points along the lines
$P\cap R, R\cap P', P'\cap R', R'\cap P$, i.e., $|\operatorname{Sing}(Y)|=5$.
Now $\Aut(Q)$ acts transitively on the configurations $Z$. Indeed, it acts 
transitively on smooth hyperplane sections $Q_0=\{t=0\} \subset Q$,
and the stabilizer of $Q_0$ surjects onto $\Aut(Q_0)$; the last statement
is clear for $\{t^2 + vz-uy\}$, for example. 

In particular, $Y$ has automorphisms leaving $P\cup R \cup P' \cup R'$
invariant but permuting the individual components via the dihedral group
of order eight.  The same action applies to nodal sextic del Pezzo surfaces,
i.e., 
$$ Y\cap \{q=0\} = P \cup P' \cup W, \quad
     Y\cap \{r=0\} = R \cup R' \cup W',$$
where $q$ and $r$ are quadratic forms vanishing on $P\cup P'$ and 
$R\cup R'$ respectively. Thus the resulting surfaces $W$ and $W'$
have the same geometric properties.
\end{rema}

\subsection{Rationality proof} \label{subsect:RP}
With these preliminaries in hand, proving rationality is routine:
\begin{theo} \label{theo:makerat}
Let $k$ have characteristic zero, $X$ a smooth cubic fourfold that contains
no planes, and $W \subset X$ a nodal sextic del Pezzo surface. 
Assume $X$ contains an algebraic two-cycle $M$ such that $M\cdot W$
is odd. Then $X$ is rational.
\end{theo}
\begin{proof}
We only need to assume that $X$ does not contain the two planes
described in Remark~\ref{rema:K3}.

Proposition~\ref{prop:residual} gives a fibration in sextic del Pezzo surfaces
$$\phi_W: \Bl_W(X) \rightarrow \bP^2.$$
A generic fiber maps to a surface with class
$$[W']=4h^2 - [W], \quad h \text{ hyperplane class},$$
as they are residual in a pencil of quadrics. 
Note that 
$$M\cdot W \equiv M\cdot W' \pmod{2}.$$

If $\phi_W^{-1}(p)$ is rational for $p \in \bP^2$ generic then $X$ is rational as well.
We focus on this generic fiber $\wW_p$.  
A sextic del Pezzo surface is rational iff it admits a 
zero-cycle of degree relatively prime to six \cite[Prop.~8]{AHTV}. 
In our situation $\wW_p$ comes with a length-two subscheme -- the pre-image
of the singularity $w_0$ -- so we only need a zero-cycle of odd degree. 
Restricting $M$ to $\wW_p$ gives what we require. 
\end{proof}

We describe the exceptional locus of the birational map
produced Theorem~\ref{theo:makerat}. Globalizing (\ref{eqn:PiBD})
over $\bP^2$, we obtain 
\[
\xymatrix{ \Pi_W:\Bl_W(X) \ar@{-->}[rr]^{\quad  \sim} \ar[rd]_{\phi_W} & &\mathcal{Q} \ar[ld] \\
							& \bP^2 & 
}
\]
where $\cQ\rightarrow \bP^2$ is a quadric surface bundle.  
Consider the Hilbert scheme of conics in fibers of $\phi_W$
$$\mathrm{Conics}(\Bl_W(X) / \bP^2) \rightarrow \bP^2$$
inducing a generically-finite morphism $S_W \ra \bP^2$ of degree three 
(see Remark~\ref{rema:blowdown}). Thus we conclude:
\begin{rema} \label{rema:getexceptional}
The birational parametrization $\bP^4 \stackrel{\sim}{\dashrightarrow}X$ of Theorem~\ref{theo:makerat}
blows up a surface birational to the triple cover
$S_W\ra \bP^2$ parametrizing conic fibrations in fibers of $\phi_W$.
\end{rema}

\section{Hodge theory and hyperk\"ahler geometry}
\label{sect:HTHG}
In this section we assume $k=\bC$.

\subsection{Statement of results} \label{subsect:SoR}
Suppose that $X$ is a smooth cubic fourfold containing a nodal sextic del Pezzo surface $W$.
The lattice of cycles generated by $W$ and the square of the hyperplane class $h$ has
intersection form
\begin{equation} \label{eqn:disc24}
\begin{array}{r|cc}
 & h^2 & W \\
 \hline
 h^2 & 3 & 6 \\
 W & 6 & 20
 \end{array}
 \end{equation}
 and thus is special of discriminant $24$ in the sense of \cite[\S 3]{Has}. 
 Note the lattice involution 
 $$h^2 \mapsto h^2, \quad W \mapsto 4h^2 - W$$
 reflecting the residuation construction of Proposition~\ref{prop:residual}.
 
We clarify precisely where Theorem~\ref{theo:makerat} applies:
\begin{theo}  \label{theo:critrat}
Let $X$ be a smooth cubic fourfold with hyperplane class $h$
admitting a saturated lattice of algebraic two-cycles
$$\left<h^2, W, M \right>$$
with intersection form
\begin{equation} \label{eqn:rank3}
\begin{array}{r|ccc}
	& h^2 & W & M \\
\hline
h^2 & 3 & 6 & m \\
W &   6 & 20 &  a  \\
M & m   & a & s
\end{array}
\end{equation}
with $a$ odd.  Then $X$ is rational.
\end{theo}
The Integral Hodge Conjecture is true for cubic fourfolds \cite[Th.~1.4]{Voisin} so we obtain:
\begin{coro}
Let $X$ be a smooth complex cubic fourfold with a saturated collection of 
Hodge classes (\ref{eqn:rank3}). Then $X$ is rational.
\end{coro}

The smooth cubic fourfolds admitting a collection of algebraic cycles (or Hodge classes) of type (\ref{eqn:disc24}) 
is an {\em irreducible} divisor $\cC_{24}$ in the moduli space $\cC$ \cite[\S 3.2]{Has}. 
Since rationality is preserved under specialization of smooth projective varieties \cite{KT}, 
Theorem~\ref{theo:critrat} follows immediately from
\begin{prop} \label{prop:dP6exist}
A generic element of $\cC_{24}$ contains nodal sextic del Pezzo surfaces.
\end{prop}
In light of Corollary~\ref{coro:dim22} and the fact that $\dim(\cC_{24})=19$, we also deduce
\begin{coro} \label{coro:dim3}
For generic $[X] \in \cC_{24}$, the Hilbert scheme of nodal sextic del Pezzo surfaces in $X$
contains two connected components $\cH_{dP}(X)$ and $\cH'_{dP}(X)$, each of dimension three and interchanged
via residuation. 
\end{coro}
The remainder of this section -- with support from Sections \ref{sect:syn} and \ref{sect:BoS} -- is devoted to the proof of Proposition~\ref{prop:dP6exist}.

\subsection{Hyperk\"ahler geometry of the variety of lines}
\label{subsect:HGVoL}
We recall constructions of \cite{BD}:  Let $X$ be a smooth cubic fourfolds with polarization $h$ and
$F_1(X)$ its variety of lines with polarization $g$. Then $F_1(X)$ is a smooth hyperk\"ahler fourfold, 
deformation equivalent to the Hilbert scheme of pairs of points on a K3 surface.  The incidence
correspondence induces a homomorphism of Hodge structures
$$\iota:H^4(X,\bZ)\rightarrow H^2(F,\bZ)(-1)$$ 
and an isomorphism on primitive cohomology
$$\iota: (h^2)^{\perp} \simeq g^{\perp}.$$
On primitive cohomology, the intersection form on $X$ coincides (up to sign) with the 
Beauville-Bogomolov form $\left(,\right)$ form on $F_1(X)$.  Note that
$\left(g,g\right)=6$ and $\left(g,H^2(F_1(X),\bZ)\right)=2\bZ$. 

For instance, (\ref{eqn:disc24}) gives intersection form 
$$
\begin{array}{r|cc}
 & h^2 & p \\
 \hline
 h^2 & 3 & 0 \\
 p & 0 & 8
 \end{array}, \quad
 p=W-2h^2,$$
hence we have Beauville-Bogomolov form
\begin{equation}\label{BBvarpi}
\begin{array}{r|cc}
 & g & \varpi \\
 \hline
 g & 6 & 0 \\
 \varpi & 0 & -8
 \end{array}, \quad \varpi=\iota(p).
 \end{equation}
The orthogonal complement complement to these divisor classes
$$\{g,\varpi \}^{\perp} \subset H^2(F_1(X),\bZ)$$
is an even lattice of discriiminant $24$ and signature $(2,19)$, isomorphic to (see \cite[\S 2,3]{Has})
$$\left( \begin{matrix} -2 & -1 \\ -1 & -2 \end{matrix} \right) \oplus \left( 8 \right) \oplus U \oplus (-E_8)^2,$$
where $E_8$ is the positive definite lattice associated with
the corresponding root system and
$$U \simeq \left( \begin{matrix} 0 & 1 \\ 1 & 0 \end{matrix} \right)$$
is hyperbolic. 
Now this lattice admits a canonical index-two extension
\begin{equation}
\left( \begin{matrix} -2 & -1 \\ -1 & -2 \end{matrix} \right) \oplus \left( 2 \right) \oplus U \oplus (-E_8)^2 \label{i2e}
\end{equation}
which is isomorphic to
\begin{equation} \label{eq:deg6lattice}
 \left( -6 \right) \oplus U \oplus U \oplus (-E_8)^2.
 \end{equation}
Indeed, fixing basis elements for the first two summands of (\ref{i2e})
$$\left(a,a\right)=\left( b,b \right)=-2, \left(a,b\right)=-1, \left(c,c\right)=2$$
then   
$$ 
w = 3 c + 2(a+b), u_2=c+a, v_2 = c+ b
$$ 
is the basis for the first two summands of (\ref{eq:deg6lattice}).  
Now (\ref{eq:deg6lattice}) is the primitive cohomology of a K3 surface of degree six.
Surjectivity of the Torelli map for K3 surfaces yields (cf. \cite[Th.~1.4]{HuyTwist}):
\begin{prop} \label{prop:24to6}
Let $X$ be a discriminant $24$ special cubic fourfold with distinguished divisors
$g,\varpi \in \Pic(F_1(X))$ as above.  Then there exists a K3 surface $S$ and a divisor $f \in \Pic(S)$ 
satisfying $f\cdot f=6$ such that
$$\{g,\varpi\}^{\perp} \hookrightarrow f^{\perp} \subset H^2(S,\bZ)
$$ 
as an index-two sublattice. The data $(S,\alpha)$, where 
$$\alpha \in H^2(S,\bZ/(1/2)\bZ)/(H^2(S,\bZ) + \bZ (f/2) )$$
determines the index-two sublattice, is a twisted K3 surface with 
Brauer class $[\alpha]\in \Br(S)[2]$.
\end{prop}
\noindent We are not asserting that $f$ is a polarization on $S$.

We return to divisors on $F_1(X)$, specifically
$$E = g+\varpi, E'=g-\varpi, \quad E^2=(E')^2 = -2$$
with Beauville-Bogomolov pairing
\begin{equation} \label{eq:Fano24}
\begin{array}{r|cc}
 & g & E \\
 \hline
 g & 6 & 6 \\
 E & 6 & -2
 \end{array}.
 \end{equation}

Assume that $X$ is generic in the divisor of cubic fourfolds of discriminant $24$.  
Moreover the main result of \cite{BHT} 
-- see also \cite[Th.~12.4]{BMinv} --  implies that $E$ and $E'$ 
are effective divisors with $\bP^1$-fibrations 
$$E,E' \rightarrow S$$
classified by $\alpha$. See \cite{HTray} for explications
of the Bayer-Macr\`i classification and descriptions of the
exceptional loci in small dimensions. 
As we shall see in Theorem~\ref{theo:gettwist}, 
$F_1(X)$ is a moduli space of $\alpha$-twisted sheaves on $S$.

Let $R$ and $R'$ denote the generic fibers of $E\rightarrow S$ and $E' \rightarrow S$ respectively.
The homology classes are readily computed via the Beauville-Bogomolov form; see \cite[\S 4]{HTmoving} for 
context. Given $v \in H^2(F_1(X),\bZ)$ we have
$$v\cdot R = \left(v,E\right), \quad v\cdot R' = \left(v,E'\right),$$
and in particular $g\cdot R = g\cdot R'=6$. Thus tracing $R$ and $R'$ through the incidence correspondence
linking $F_1(X)$ and $X$, we obtain sextic ruled surfaces
$$T, T' \subset X,$$
each deforming in family parametrized by $S$.  
Note that
$$[T']=4h^2 - [T]$$
in the cohomology of $X$.  

\section{Synthetic geometry}
\label{sect:syn}
\subsection{Geometry of the scrolls} \label{subsect:GotS}
While hyperk\"ahler theory tells us the sextic scrolls exist, a direct construction
clarifies what they look like.  
The double point formula guarantees that $T$ and $T'$
are singular; {\em assuming} they have only nodes (transverse double points)
then each has two singularities (see \cite[\S 7]{HTRC}).  This is in fact the case:

\begin{prop} \label{prop:getscroll}
Let $X$ be a discriminant $24$ special cubic fourfold that is generic with this property. 
Then $X$ admits two families $\cH_{ss}(X)$ and $\cH'_{ss}(X)$ of sextic scrolls, generically with two nodes; each family is parametrized by
the K3 surface $S$ described in Proposition~\ref{prop:24to6}.  
\end{prop}
\noindent These surfaces are called {\em two-nodal sextic scrolls}.  

Most of this is already proven except for the explicit construction of the families:
\begin{lemm} \label{lemm:makesextic}
Let $\wT\simeq \bP^1 \times \bP^1 \hookrightarrow \bP^7$ be a smooth sextic scroll associated
with the linear series $\Gamma(\cO_{\bP^1 \times \bP^1}(1,3))$.  Choose four generic points
$$t_{1+},t_{1-},t_{2+},t_{2-} \in \bP^1 \times \bP^1$$
and a generic line $\ell \subset \bP^7$ such that  
$$\ell \cap \Sec(t_{1+},t_{1-}) \neq \emptyset, \ \ell \cap \Sec(t_{2+},t_{2-})\neq \emptyset.$$
Projecting from $\ell$
$$\pi_{\ell}: \wT \rightarrow T \subset \bP^5$$
yields a sextic scroll $T$ with nodes $t_1$ and $t_2$. This surface is contained in cubic fourfolds
and the generic such variety is smooth.  
\end{lemm}
The argument is very similar to the corresponding statement for nodal sextic del Pezzo surfaces.
It can also be obtained by direct computation for an example over a convenient finite field.
For our purposes, it is most useful to outline the key properties:
\begin{itemize}
\item{$\dim \Gamma(\cI_{T}(2))=2$ and $\dim \Gamma(\cI_{T}(3))=18$;}
\item{$T \subset Y$ a complete intersection of two quadrics;}
\item{the ideal of $T$ is generated in degrees $\le 3$
$$\cI_{T}=\left< Q_0,Q_1, F_0,F_1,F_2,F_3,F_4,F_5\right>$$
hence a generic cubic fourfold $X \supset T$ is smooth;}
\item{fixing $X\supset T$, the intersection
$$X \cap Y = T \cup T'$$
where $T'$ is also a two-nodal sextic scroll.}
\end{itemize}
As a dividend, we obtain
\begin{coro} \label{coro:dim19}
The moduli space of incidences
$$\cU = \{(T,X): \text{ two-nodal sextic scroll } T \subset X
\text{ cubic fourfold} \}
$$
is unirational of dimension $2+2+17=21$ and has a natural involution 
$$
\begin{array}{rcl}
\cU & \dashrightarrow & \cU \\
T & \mapsto & T'.
\end{array}
$$
\end{coro}
\subsection{Relating the scrolls and del Pezzos} \label{subsect:relSdP}
Proposition~\ref{prop:dP6exist} involves a synthetic construction
for the nodal sextic del Pezzo surfaces starting from the two-nodal sextic scrolls.  
This suffices because discriminant $24$ cubic fourfolds generically contain
the latter surfaces by Proposition~\ref{prop:getscroll}.
We carry out this construction in several stpes.

\subsubsection*{First construction}
Each sextic del Pezzo surfaces $\wW \subset \bP^6$ admits three conic
bundle structures (see (\ref{eqn:3conics})):
$$\gamma_i: \wW \rightarrow \bP^1, \quad C_i = \gamma_i^*\cO_{\bP^1}(1).$$
Each has a factorization through a $\bP^2$ bundle
$$
\xymatrix{ 
\wW \ar@{^{(}->}[r] \ar[rd]_{\gamma_i} & \wV_i \ar[d]^{\varpi_i} \\
			                & \bP^1
}
$$
where $\wV_i=\bP(({\gamma_i}_*(\omega_{\gamma_i}^{\vee}))^{\vee})$. 
The Leray spectral sequence gives
$$\Gamma(\wW, \omega_{\gamma_i}^{\vee}(nC_i))=
\Gamma(\bP^1,{\gamma_i}_* (\omega_{\gamma_i}^{\vee})\otimes \cO_{\bP^1}(n)).$$
Since
$$\Gamma(\wW, \omega_{\gamma_i}^{\vee}(nC_i))=\Gamma(\wW,\cO_{\wW}((n-2)C_i)\otimes \cO_{\wW}(1))\simeq k^{3n+1}$$
we find that
$${\gamma_i}_*(\omega_{\gamma_i}^{\vee})= \cO_{\bP^1}\oplus \cO_{\bP^1}(-1)^2.$$
Thus 
$$\wV_i \simeq \bP(\cO_{\bP^1}(1)^2 \oplus \cO_{\bP^1})
\simeq \bP(\cO_{\bP^1}(-1)^2 \oplus \cO_{\bP^1}(-2))$$
and we have a factorization
$$\wW \subset \wV_i \subset \bP^6$$
through a quartic threefold ruled in planes.  

Writing $\xi = [\cO_{\wV_i}(1)]$ and $f$ for the class of a fiber of $\varpi_i$, we compute
$$\xi^3=4, \ \xi^2f=1,\ [\wW]=2\xi-2f.$$

\subsubsection*{A computation}
Let $\pi_x: \wW \rightarrow \bP^5$ denote the projection from a point $x \in \Sec(w_+,w_-)$
as in Section~\ref{subsect:NDPS}. 
\begin{prop}
There exists a unique section $\sigma$ of $\varpi_i$ with the following properties:
\begin{itemize}
\item{$\sigma(\bP^1) \subset \wV_i \subset \bP^6$ is a conic;}
\item{$\sigma(\gamma_i(w_{\pm}))=w_{\pm}$;}
\item{the projection
$$\pi_x:\wV_i \rightarrow V_i$$ 
pinches $\wV_i$ along $\sigma(\bP^1)$ via an involution of the 
conic gluing $w_+$ and $w_-$.}
\end{itemize}
\end{prop}
\begin{proof}
For concreteness we set $i=1$.  
In Section~\ref{subsect:NDPS} we assumed that $\gamma_1(w_+)\neq \gamma_1(w_-)$, so it is possible for a 
section to contain both of them. The curves $E_1$ and $\ell_{23}$ are sections of $\gamma_1$ spanning
the distinguished line-subbundle $\bP(\cO_{\bP^1}^2)\subset \wV_1$. Our assumptions also guarantee
that $w_+$ and $w_-$ are outside this line. 
The sections with image of degree two are given by 
$$\cO_{\bP^1}(-2) \hookrightarrow \cO_{\bP^1}(-1)^2 \oplus \cO_{\bP^1}(-2),$$
indexed by elements of $\bP(\Gamma(\cO_{\bP^1}(1)^2 \oplus \cO_{\bP^1}))\simeq \bP^4$.
We may rescale so that the coordinates of $w_+$ and $w_-$ are $1$, i.e.
$$w_+=[a_{1+},a_{2+},1], \quad w_-=[a_{1-},a_{2-},1].$$
Linear algebra gives a unique section the desired property. 

The double-point formula \cite[Th.~9.3]{Fulton} applied to the projection
$$\pi_x:\wV_1 \rightarrow \bP^5$$
gives the class of the locus in $\wV_1$ sitting over the singularities of the image:
It is the residual intersection, to two lines in fibers of $\varpi_1$, in a 
codimension-two linear section of $\wV_1 \subset \bP^6$; its class equals to the
class of $\sigma(\bP^1)$.  
We compute that 
$$\sigma(\bP^1)\cap \wW = \{w_+,w_-\}.$$

Let $\Pi$ denote the span of $\sigma(\bP^1)$ in $\bP^6$. Since $x \in \Sec(w_+,w_-)$,
it also sits on $\Pi$. Projection from $x$ induces a degree-two morphism
$$\sigma(\bP^1) \rightarrow \Sing(V_1),$$
to the line in $V_1$ along which it is singular.
Moreover, we have
$$w_0=\pi_x(w_{\pm}) = \pi_x(\sigma(\bP^1)\cap \wW) \in \Sing(V_1).$$
\end{proof}
\begin{rema}
While $V_1 \subset \bP^5$ is a threefold of degree four containing $W$, it
is not a complete intersection of two quadrics! Indeed, 
$$\dim \Gamma(\cI_{V_1}(2))=1$$
as codimension-two linear sections of $V_1$ are rational normal quartic
curves in $\bP^3$.
\end{rema}

Suppose $W\subset X$, a cubic fourfold. The conics in $W$ arising as fibers of $\gamma_i$ are residual to 
lines in $X$.  Computing in $\wV_i$ we find
$$X\cap \wV_i = \wW \cup \wT_i, \quad [\wT_i]=\xi+2f,$$
where $\wT_i \subset \bP^6$ is a sextic scroll with
\begin{equation} \label{eqn:4pts}
\wT_i \cap \sigma(\bP^1) = \{w_+,w_-, t_{2+},t_{2-}\}.
\end{equation}
Projection $\pi_x$ takes
$$\wT_i \rightarrow T_i$$
with nodes at the images of (\ref{eqn:4pts}).

\begin{prop}\label{prop:getsexticruled}
Let $X$ be a generic cubic fourfold containing a nodal sextic del Pezzo surface $W$ and
$$W \subset V_i \subset \bP^5$$
the threefold spanned by the conics from $|C_i|$. We have a residual intersection
$$X\cap V_i = T_i \cup W$$
where $T_i$ is a sextic scroll with two nodes $t_1:=w_0$ and $t_2$.  
\end{prop}    
\begin{coro} \label{coro:threescrolls}
Each $W\subset X$ admits {\em three} residual two-nodal sextic scrolls $T_1,T_2,$ and $T_3$, indexed by the conic
fibrations on $\wW$.   
\end{coro}

\subsubsection*{Reversing the construction}
We return to the situation of Lemma~\ref{lemm:makesextic}: 
$\wT\simeq \bP^1 \times \bP^1 \hookrightarrow \bP^7$ is a smooth sextic scroll, with four generic points
$$t_{1+},t_{1-},t_{2+},t_{2-} \in \bP^1 \times \bP^1$$
and a line $\ell \subset \bP^7$ such that  
$$\ell \cap \Sec(t_{1+},t_{1-}) \neq \emptyset, \ \ell \cap \Sec(t_{2+},t_{2-})\neq \emptyset.$$

Now choose $p \in \ell$, not on the secant lines, so we have a factorization
$$\pi_{\ell}: \bP^7 \stackrel{\pi_p}{\dashrightarrow} \bP^6 \stackrel{\pi_{p'}}{\dashrightarrow} \bP^5$$
where $p'=\pi_p(\ell)$.  
From the perspective $\wT$, the point $p$ is not special: A generic point of $\bP^7$ is contained 
in the span of some pair of secant lines for $\wT$. 
Now the projected scroll is contained in a distinguished three-dimensional degree-four scroll
$$\pi_p(\wT) \subset \Sigma_p \rightarrow \bP^6$$
where the fibrations $\wT \rightarrow \bP^1$ and $\Sigma_p \rightarrow \bP^1$ are compatible;
see Theorem~\ref{theo:uniquescroll} for an explanation why. 
Projecting from $p'$ yields
$$T \subset \pi_{p'}(\Sigma_p).$$
If $X$ is a cubic fourfold containing $T$ then residual intersection 
yields a nodal sextic del Pezzo surface
$$X \cap \Sigma_p = T \cup W_p.$$
We summarize this as follows:
\begin{prop} \label{prop:getNSdP}
Let $X$ be a generic cubic fourfold containing a two-nodal sextic scroll $T$.
If $\ell$ is the line of projection for $T$ then for generic $p\in \ell$ we obtain
a nodal sextic del Pezzo surface $W_p \subset X$, residual to $T$ in the
projection of a degree-four three-dimensional scroll projected to $\bP^5$.
\end{prop}
This completes the proof of Proposition~\ref{prop:dP6exist}, assuming the 
constructions of Section~\ref{sect:BoS}. 

\subsection{Speculation on the correspondence} \label{subsect:spec}
Let $X$ be a generic cubic fourfold of discriminant $24$,
$\cH_{dP}(X)$ and $\cH'_{dP}(X)$ the families of nodal sextic del Pezzo surfaces in $X$,
and $\cH_{ss}(X)$ and $\cH'_{ss}(X)$ the families of two-nodal sextic scrolls in $X$. 
Assume that 
$$[W]=4h^2 - T$$
for $W \in \cH_{dP}(X)$ and $T \in \cH_{ss}(X)$ via the residuation
constructions of Propositions~\ref{prop:getsexticruled} and \ref{prop:getNSdP}.

Recall that
\begin{enumerate}
\item{given $W \in \cH_{dP}(X)$, the $W' \in \cH'_{dP}(X)$ directly residual to $W$
are parametrized by a $\bP^2$ and {\em vice versa} (Prop.~\ref{prop:residual});}
\item{$\cH_{dP}(X)$ and $\cH'_{dP}(X)$ are three-dimensional (Cor.~\ref{coro:dim3});}
\item{$\cH_{ss}(X)$ and $\cH'_{ss}(X)$ are isomorphic to the degree-six K3 surface $S$
(Props.~\ref{prop:24to6} and \ref{prop:getscroll});}
\item{given $W \in \cH_{dP}(X)$, we have three $T_1,T_2,T_3 \in \cH_{ss}(X)$
directly residual to $W$ (Cor.~\ref{coro:threescrolls});}
\item{given $T \in \cH_{ss}(X)$, the $W \in \cH_{dP}(X)$ directly residual to $T$
are parametrized by a $\bP^1$ (Prop.~\ref{prop:getNSdP}).}
\end{enumerate}
Consider items (1) and (2) and the residuation correspondence
$$\cR =\{(W,W'): W\cup W' = \{Q_{\lambda}=0\}_{\lambda \in \bP^1} \cap X\}
\subset \cH_{dP}(X) \times \cH'_{dP}(X).$$
Fixing $W \in \cH_{dP}(X)$, let
$\cR_W \simeq \bP^2$ denote the elements directly residual to $W$. For generic
$$W'_t, t\in \cR_W$$
we have 
$$W \in \cR_{W'_t}\simeq \bP^2.$$
In other words, each $W$ is a contained a family of divisors parametrized by $\bP^2$, 
with the generic member isomorphic to $\bP^2$. 
This resembles the duality correspondence
$$\{(p,L): L(p)=0\} \subset \bP^3 \times (\bP^3)^{\vee}.$$

The K3 surface appearing in item (3) naturally embeds
$$S \subset Q \subset \bP^4$$
where $Q$ is a smooth quadric threefold. Lines $\ell \subset Q$ are three-secant to $S$,
inducing
\begin{align*}
\bP^3 \simeq F_1(Q) & \hookrightarrow S^{[3]} \\
		    \ell & \mapsto \ell \cap S
\end{align*}
which resembles the correspondence in item (4). 
Furthermore, through each point $s\in S$, there is a conic parametrizing the lines
$$s \in \ell \subset Q,$$
i.e. the projective tangent cone of $T_sQ \cap Q$. 
This is similar to what we observe in item (5).

Further, for each line $\ell \subset Q$, projection induces
$$\Bl_{\ell\cap S}(S) \rightarrow \bP^2$$
a degree-three cover.  On the other hand, conics
in fibers of
$$\phi_W: \Bl_W(X) \rightarrow \bP^2$$
give a degree-three morphism $S_W\rightarrow \bP^2$.  
When $\phi_W$ has a section, the resulting birational parametrization
of $X$ blows up a surface birational to $S_W$ by Remark~\ref{rema:getexceptional}.

\begin{ques} \label{ques:HdP}
Let $X$ be a generic cubic fourfold of discriminant $24$ and $S\subset \bP^4$ the degree-six
K3 surface from Prop.~\ref{prop:24to6}. Is $\cH_{dP}(X)$ isomorphic to $\bP^3$,
realized as the lines $\ell$ three-secant to $S$? 
Can we identify the surfaces $\Bl_{\ell \cap S}(S)$ and $S_W$?
\end{ques}

\section{Background on scrolls} \label{sect:BoS}
We work over a field $k$ of characteristic zero.

Fix positive integers $a$ and $r$ and consider the vector bundle
$$\cE = \cO_{\bP^1}(-a)^s \oplus \cO_{\bP^1}(-a-1)^{r-s}$$
for some $s=1,\ldots,r$.  The dual bundle $\cE^{\vee}$ is generated 
by
$$s(a+1)+(r-s)(a+2)= r(a+2) - s $$
global sections,
yielding
$$\cE \hookrightarrow \Gamma(\bP^1,\cE^{\vee})^{\vee} \otimes \cO_{\bP^1} \simeq  \cO_{\bP^1}^{r(a+2)-s}.$$
Its projectivization
$$\Sigma := \bP(\cE) \subset  \bP(\Gamma(\bP^1,\cE^{\vee})^{\vee})\simeq \bP^{r(a+2)-s-1}=\bP^n, \quad n:=r(a+2)-s-1,$$
is a {\em generic scroll} of dimension $r$ and degree 
$$d:=as+(r-s)(a+1)=r(a+1)-s.$$
Here ``generic'' refers to the splitting type of $\cE$.  
Since $r+d=n+1$ the variety $\Sigma$ is of minimal degree.   

Automorphisms of $\cE$ take the form
$$\left( \begin{matrix} C_1 & L \\	
				  0 & C_2 \end{matrix} \right)$$
where $C_1$ is an $s\times s$ invertible matrix of constant forms, $C_2$ an $(r-s)\times(r-s)$ invertible matrix of constant
forms, and $L$ is an $s \times (r-s)$ matrix of linear forms.  
Thus 
$$\dim(\Aut(\cE))=s^2 + (r-s)^2 + 2s(r-s) = r^2$$
and 
$$\dim(\Aut(\Sigma))=r^2+2.$$
The Hilbert scheme of generic scrolls
\begin{equation} \label{Hilb1}
\Hilb(\Sigma \subset \bP^n)
\end{equation}
thus has dimension $n^2+2n-r^2-2$.  

Now assume $r\ge 2$ so that
$$\Pic(\Sigma) = \bZ f + \bZ h$$
where $h$ is the hyperplane class and $f$ is the fiber of $\Sigma=\bP(\cE) \rightarrow \bP^1$.  
The space
$$\Gamma(\Sigma,\cO_{\Sigma}(h+2f))=\Gamma(\bP^1,\cE^{\vee}(2))=\Gamma(\bP^1,\cO_{\bP^1}(a+2)^s \oplus \cO_{\bP^1}(a+3)^{r-s})$$
has dimension $r(a+4)-s=n+2r+1$. The Hilbert scheme of flags
\begin{equation} \label{Hilb2}
\Hilb(\Xi \subset \Sigma \subset \bP^n), \quad [\Xi]=h+2f,
\end{equation}
a $\bP^{n+2r}$ bundle over (\ref{Hilb1}), 
has dimension 
$$n^2+3n-r^2+2r-2.$$
The number of moduli -- modding out by the action of $\PGL_{n+1}$ -- is
\begin{equation} \label{moduli1}
\#\Moduli(\Xi \subset \Sigma \subset \bP^n) = n - r^2 + 2r -2.
\end{equation}

Assume that $\Xi$ is smooth. We represent $\Xi=\bP(\cF)$ where $\cF$ is a rank-$(r-1)$ vector bundle.  
The equation for $\Xi$, a section 
$$s_{\Xi} \in \Gamma(\Sigma,\cO_{\Sigma}(h+2f)) \simeq \Hom_{\bP^1}(\cO_{\bP^1}(-2),\cE^{\vee}),$$
yields exact sequences
\begin{align*}
 0 \ra \cO_{\bP^1}(-2) \rightarrow & \cE^{\vee} \rightarrow \cF^{\vee} \rightarrow 0, \\
 0 \ra \cF \rightarrow & \cE \rightarrow \cO_{\bP^1}(2) \rightarrow 0. 
 \end{align*}
 In particular, 
 $$\deg(\cF^{\vee})=\deg(\cE^{\vee})+2=r(a+1)-s+2, \quad
 \chi(\cF^{\vee})=\chi(\cE^{\vee})+1.$$
 
 \begin{exam}
 
 Suppose $r=s=2$ so that
 $$\cE^{\vee} = \cO_{\bP^1}(a)^2.$$
 For generic $\Sigma$ in the linear series $|h+2f|$ we have
 $$\cF^{\vee} = \cO_{\bP^1}(2a+2)=\cO_{\bP^1}(n+1).$$
When $s=1$ so that
 $$\cE^{\vee} = \cO_{\bP^1}(a) \oplus \cO_{\bP^1}(a+1)$$
 then generically
 $$\cF^{\vee} = \cO_{\bP^1}(2a+3)=\cO_{\bP^1}(n+1).$$
 
 Now take $r=3$. For $s=3$ we have $\cE^{\vee} = \cO_{\bP^1}(a)^3 $
 and generically
 $$\cF^{\vee} = \begin{cases} \cO_{\bP^1}(1 + 3a/2)^2 & \text{ if $a$ even} \\
 					     \cO_{\bP^1}((3a+1)/2) \oplus \cO_{\bP^1}((3a+3)/2)  &\text{ if $a$ odd.}
					     \end{cases}
					     $$
When $s=2$, $\cE^{\vee}=\cO_{\bP^1}(a)^2 \oplus \cO_{\bP^1}(a+1)$
and generically
$$\cF^{\vee} = \begin{cases} \cO_{\bP^1}(1+3a/2) \oplus \cO_{\bP^1}(2+3a/2) & \text{ if $a$ even } \\
					     \cO_{\bP^1}((3a+3)/2)^2 & \text{ if $a$ odd.}
		        \end{cases}
					     $$
When $s=1$, $\cE^{\vee}=\cO_{\bP^1}(a) \oplus \cO_{\bP^1}(a+1)^2$
and generically
$$\cF^{\vee} = \begin{cases} \cO_{\bP^1}(2+3a/2)^2  & \text{ if $a$ even } \\
					     \cO_{\bP^1}((3a+1)/2)\oplus \cO_{\bP^1}((3a+3)/2)  & \text{ if $a$ odd.}
		        \end{cases}
					     $$
\end{exam}

We have seen that $\Gamma(\bP^1,\cE^{\vee}) \subset \Gamma(\bP^1,\cF^{\vee})$ has codimension one.
Thus $\Xi\subset \bP^n$ is not linearly normal but arises as the projection of a scroll 
$$\wXi = \bP(\cF) \subset \bP(\Gamma(\bP^1,\cF^{\vee})^{\vee})\simeq \bP^{n+1}$$
from a point $p$; recall that
$$2\dim(\wXi) = 2(r-1) < 2r + ar - s - 1 = \dim(\bP^n)$$
so the projection to $\bP^n$ is smooth for generic $p$.  
The Hilbert scheme 
 \begin{equation} \label{Hilb3}
\Hilb(\wXi \subset \bP^{n+1}\ni p),
\end{equation}
has dimension
$$
(n+1)^2+2(n+1)-(r-1)^2-2 
+(n+1) = n^3 + 5n - r^2 + 2r + 1.$$
Modding out by the action of $\PGL_{n+2}$ we get
\begin{equation} \label{moduli2}
\# \Moduli(\wXi \subset \bP^{n+1}\ni p) = n - r^2 + 2r -2.
\end{equation}

The equality between (\ref{moduli1}) and (\ref{moduli2}) suggests analyzing
fibers of the morphism
$$\Phi: \Hilb(\Xi \subset \Sigma \subset \bP^n)/\PGL_{n+1} 
\rightarrow \Hilb(\wXi \subset \bP^{n+1}\ni p)/\PGL_{n+2},
$$ 
which we expect to be generically finite.  
\begin{theo} \label{theo:uniquescroll}
Let $\wXi \subset \bP^{n+1}$ be a generic scroll of dimension $r-1 < n/2$
and $p \in \bP^{n+1}$ a generic point. Then the projection
$$\pi_p: \wXi \rightarrow \Xi \subset \bP^n$$
is contained in a unique scroll of dimension $r$.  
In particular, $\Phi$ is birational.  
\end{theo}
\begin{proof}
Given $\wXi=\bP(\cF) \subset \bP^{n+1}$, projections from
$p\in \bP^{n+1}$ correspond to elements $\lambda \in \Gamma(\bP^1,\cF^{\vee})^{\vee}$. 
However, Serre duality gives
$$\Gamma(\bP^1,\cF^{\vee})^{\vee} = \Ext^1(\cF^{\vee},\omega_{\bP^1}) = \Ext^1(\cF^{\vee},\cO_{\bP^1}(-2))$$
hence an extension
$$0 \ra \cO_{\bP^1}(-2) \ra \cE^{\vee} \ra \cF^{\vee} \rightarrow 0.$$

Since $\cF^{\vee}$ has generic splitting type, the same holds for $\cE^{\vee}$ for generic $p$.
Repeating the analysis above, we deduce that 
$$\Xi \subset \bP(\Gamma(\bP^1,\cE^{\vee})^{\vee})$$
and is associated with a section of $\Hom(\cO_{\bP^1}(-2),\cE^{\vee})=\Gamma(\bP^1,\cE^{\vee}(2)).$
\end{proof}

\begin{exam} cf.~\cite{EParkcurves}
Suppose that $\wXi \subset \bP^{n+1}$ is a rational normal curve.  
For odd $n=2a+1$, a generic $p \in \bP^{2a+2}$ lies on a curve of $(a+2)$-secant 
$(a+1)$ planes, parametrized by $\bP^1$ \cite[\S 1.5]{RanSch}.
For even $n=2a+2$, a generic $p \in \bP^{2a+3}$ lies on a $(a+2)$-secant $(a+1)$ plane,
which is unique by a classical result of Sylvester \cite[\S 0.4]{RanSch}.   

In the even case, let $\xi_1+\cdots+\xi_{a+2} \in \Sym^{a+2}(\wXi)$ denote
the divisor such that $p \in \spa(\xi_1,\ldots,\xi_{a+2})$. The corresponding
projected points 
$$\xi_1,\ldots,\xi_{a+2} \in \Xi \subset \bP^{2a+2}$$
span a $\bP^a$. The Kapranov construction \cite{KapJAG} gives a canonical
rational normal curve
$$\xi_1,\ldots,\xi_{a+2} \in R \subset \bP^a$$
and an isomorphism $\iota: R \stackrel{\sim}{\ra} \Xi$ respecting $\xi_1,\ldots,\xi_{a+2}$.
The join of $R$ and $\Xi$ via $\iota$ is isomorphic to
$$\bP(\cO_{\bP^1}(-a) \oplus \cO_{\bP^1}(-a-1)),$$
with $R$ the distinguished section and $[\Xi]=R+(a+3)f=h+2f$, where $f$ is the fiber over $\bP^1$ and
$h$ is the hyperplane class. Hence $\Sigma$ is obtained as this join.  

In the odd case, let $(\xi_1+\cdots +\xi_{a+2})_t, t\in \bP^1,$ denote the divisors
with secants containing $p$.  The corresponding linear series has graph
$$R \subset \wXi \times \bP^1 \simeq \bP^1 \times \bP^1$$
with bidegree $(a+2,1)$.  
After projection to $\bP^{2a+1}$, each divisor $(\xi_1+\cdots +\xi_{a+2})_t$ spans a
$\bP^a_t$.  Putting this all together, we obtain
$$\Sigma:=\bP^1 \times \bP^1 \subset \bP^a \times \bP^1 \subset \bP^{2a+1}$$
with $\Sigma$ embedded via forms of bidegree $(a,1)$.  
We again have $[\Xi]=h+2f$ on $\Sigma$.  
\end{exam}

We refer the reader to \cite{Nagel,HKP} for further analysis of these constructions.

\section{Twisted sheaf interpretations}\label{sect:TSI}
In this section, the base field $k=\bC$. Our goal is to explain why
$F_1(X)$ is isomorphic to a moduli space of $\alpha$-twisted sheaves on $(S,f)$.

\subsection{Twisted cohomology}
We follow \cite{HuySte,HuySte2}.

Let $S$ be a projective K3 surface with Mukai lattice
$$\wH(S,\bZ) = \rH^0(S,\bZ) \oplus \rH^2(S,\bZ) \oplus \rH^4(S,\bZ)$$
with bilinear form
$$\left<(r_1,D_1,s_1),(r_2,D_2,s_2)\right> = D_1\cdot D_2 - r_1s_2-r_2s_1.$$
Let 
$$\alpha \in \Br(S)=\rH^2(S,\bQ)/ \left( \Pic(S)_{\bQ}+\rH^2(S,\bZ) \right)$$ 
and $B\in \rH^2(S,\bQ)$ mapping to $\alpha$.  Write
$$\wH(S,B,\bZ) =\wH(S,\bZ)$$
as a lattice with weight-two Hodge structure on $\wH(S,\bC)$ induced by
$$\exp(B)\omega=\omega+(B\wedge \omega), \quad \omega \in \rH(\Omega^2_S).$$
The same lattice may be regarded as a mixed Hodge structure 
$\rH^*(S,B,\bZ)$ with weight filtration
$$W_i := \exp(B) \oplus_{j\le i}\rH^j(S,\bQ).$$
Applying $\exp(-B)$, we may understand
$$\exp(-B): \wH(S,B,\bZ) \hookrightarrow \wH(S,\bQ)$$
as a sub Hodge structure.  
Primitive Mukai vectors refer to vectors
primitive with respect to this twisted cohomology.

The bilinear form is invariant under the action via $\exp(B)$:
$$\exp(B)(r,D,s) = (r,D+Br,s+B\wedge D+\frac{B^2r}{2})$$
thus
\begin{align*}
\left<\exp(B)(r,D,s),\exp(B)(r,D,s)\right> &= D^2+2B\wedge Dr
+B^2r^2\\
& \quad \quad - 2r(s+B\wedge D+B^2r/2)\\
&=D^2-2rs=\left<(r,D,s),(r,D,s)\right>.
\end{align*}

\begin{exam}
Take $B=f$ a polarization.  Then the weight-zero part is generated
by $(I+f+\frac{f^2}{2})\cdot 1$ and the weight-two part is obtained by adding
$$\{\gamma + f\wedge \gamma: \gamma \in \rH^2(S,\bZ)\}.$$

For Mukai vectors associated with moduli spaces of simple sheaves
$\cE \rightarrow S$,
i.e.,
$$(r,D,s)=v(\cE),$$
we have
$$\exp(f)v(\cE) = v(\cE\otimes \cO_S(f)).$$
When our focus is on projective bundles $\bP(\cE)\ra S$, it is natural
to consider orbits of Mukai vectors under exponentiation by
line bundles.  
\end{exam}
\begin{rema} \label{rema:twistHodge}
For general $B$, twisted Hodge classes in $\wH(S,B,\bZ)$ include 
$$
[pt] \in \rH^4(S,\bZ),   \Pic(S) \subset \rH^2(S,\bZ),  r\left(1+B+\frac{B^2}{2}\right),
$$
where $r$ is a positive integer such that $rB$ and $rB^2/2$ are 
integral. 
\end{rema} 
We refer the reader to \cite[\S 2]{HuySte} for the general analysis of the twisted Picard
group $\Pic(S,B)$.

\subsection{Moduli results}
Existence theorems for moduli of twisted sheaves
are due to Yoshioka \cite{yoshioka}. He takes, as input, a smooth Brauer-Severi fibration 
$$p:Y \rightarrow S$$
representing $\alpha \in \Br(S)$. 
Let $\Coh(S,Y)$ denote coherent sheaves on $Y$ that are
locally, in the analytic or \'etale topology on $S$, tensor
products of coherent sheaves pulled-back
from $S$ and line bundles.  There is a distinguished
vector bundle $\cG$ on $Y$ corresponding to the non-trivial
extension
$$0 \rightarrow \cO_Y \rightarrow \cG \rightarrow T_{Y/S} \rightarrow 0.
$$
A coherent sheaf $\cF$ on $Y$ belongs to $\Coh(S,Y)$
iff the ``global generation map''
$$p^*p_*(\cG^{\vee} \otimes \cF) \ra \cG^{\vee} \otimes \cF$$
is an isomorphism \cite[Lem.~1.5]{yoshioka}.
The minimum of
$$\{ \operatorname{rk}(\cF)>0: \cF \in \Coh(S,Y) \}$$
is the order of $\alpha$.  
We recall the basic existence results:
\begin{itemize}
\item{Assume that $v$ is a primitive Mukai vector and $H$
is a general polarization with respect to $v$. Then all $\cG$-twisted
semi-stable sheaves $\cF$ with $v_G(\cF)=v$ are $G$-twisted
stable. Thus $M^{Y,\cG}_H(v)$ is a projective manifold
when it is non-empty \cite[Thm.~3.11]{yoshioka}.}
\item{Assume in addition that $r>0$ and $\left<v,v\right>\ge -2$.
Then $M^{Y,\cG}_H(v)$ is non-empty and deformation equivalent to a punctual Hilbert scheme of a K3 surface.
When $\left<v,v\right>=0$ then it is a K3 surface  \cite[Thm.~3.16]{yoshioka}.}
\item{The second cohomology of the associated moduli space is
isomorphic to $v^{\perp}$ or $v^{\perp}/\bZ v$ depending
on whether $\left<v,v\right>=0$ or $\neq 0$ 
\cite[Thm.~3.19]{yoshioka}.
}
\end{itemize}
Below we will write $M_v(S,\alpha)$ for the moduli space 
$M^{Y,\cG}_H(v)$, suppressing the polarization and choice of
representative for $\alpha$.  

\subsection{Application to the variety of lines} \label{subsect:gettwist}
Fix an isomorphism
$$\rH^2(S,\bZ)\simeq U^{\oplus 3} \oplus (-E_8)^{\oplus 2}$$
where 
$$U \simeq \left< u_i,v_i\right> =\left( \begin{matrix} 0 & 1 \\ 1 & 0 \end{matrix} \right), \quad i=1,2,3.$$

Let $(S,f)$ have degree six;
we choose the lattice isomorphism so that
$$f = u_1+3v_1,$$
in the first $U$ factor. The discriminant group of $f_6^{\perp}$
is generated by $(u_1-3v_1)/6$. 

Set $B=\frac{v_1+u_2-v_2}{2}$ so that $B^2=-[pt]/2$ and $B\wedge f = [pt]/2$;
as before, $\alpha$ denotes the corresponding element of $\Br(S)[2]$.  We
obtain (see Remark~\ref{rema:twistHodge}) integral Hodge classes
$$\{2-(v_1+u_2-v_2), f, [pt]\}$$
and intersection matrix
$$\left( \begin{matrix} -2 & -1 & -2 \\	
				 -1 & 6 & 0 \\
				 -2 & 0 & 0 \end{matrix} \right).$$
Consider the Mukai vector
$$v=2-(v_1+u_2-v_2)+f = 2+ u_1+2v_1 - u_2+v_2$$
whose orthogonal complement has elements
$$g=-6+3(v_1+u_2-v_2) - f + 2[pt] = -6  - u_1 + 3u_2 - 3v_2 + 2[pt]$$
and
$$E=2-(v_1+u_2-v_2) + f + [pt] = 2 + u_1 + 2v_1 -u_2 + v_2 + [pt].$$
Since we have
$$\begin{array}{r|cc}
   & g & E\\
\hline
g & 6 & 6 \\
E & 6 & -2
\end{array},
$$
which coincides with (\ref{eq:Fano24}),
the moduli space of twisted sheaves $M_v(S,\alpha)$ has Picard group isomorphic to the Picard group $F_1(X)$.
We refer to \cite[Thm.~5.7(a),\S 8]{BMinv} for details on the birational geometry of such moduli spaces.
The divisor $E$ may be interpreted as the projectivization of a rigid rank-two $\alpha$-twisted sheaf on $S$.

We summarize this discussion:
\begin{theo} \label{theo:gettwist}
Let $F_1(X)$ denote the variety of lines on a cubic fourfold of discriminant $24$. 
Then $F_1(X)$ is birational to a moduli space $M_v(S,\alpha)$
where $S$ is a K3 surface, $f\in \Pic(S)$ with $f^2=6$, and $\alpha \in \Br(S)[2]$ and $v$ are
as specified above.
\end{theo}
\begin{proof}
This follows from the Torelli Theorems for hyperk\"ahler fourfolds of $K3^{[2]}$ type 
\cite[Thm.~9.8]{MarkmanTorelli} and for polarized K3 surfaces.  
\end{proof}
The classification of Brauer classes $\alpha \in \Br(S)[2]$ -- up to monodromy of 
the polarized surface $(S,f)$ --
is given in \cite[Prop.~9.2]{vanGeemen}.  There are three orbits, characterized via
discriminant quadratic forms.

\begin{ques}
Starting from the \'etale $\bP^1$ bundle $E\rightarrow S$, restriction of scalars yields a $(\bP^1)^3$ bundle
$$\mathbf{R}_{Z/S^{[3]}}E \rightarrow S^{[3]};$$ 
here $Z \ra S^{[3]}$ is the universal subscheme.  Restricting to the distinguished $\bP^3 \subset S^{[3]}$
(see Question~\ref{ques:HdP}) gives
$$\mathbf{R}_{Z/S^{[3]}}E \times_{S^{[3]}} \bP^3 \rightarrow \bP^3.$$
A relative hyperplane section is a fibration in sextic del Pezzo surfaces. Is this 
related to the universal family
$\wcW \rightarrow \cH_{dP}$
of normalized sextic del Pezzo surfaces in $X$?
\end{ques}

\subsection{Higher-dimensional hyperk\"ahler manifolds}
\label{subsect:HDHM}
Given a cubic fourfold $X$, the variety of lines $F_1(X)$ is the first of an
infinitely sequence of hyperk\"ahler manifolds associated with $X$ \cite[\S 29]{BLMNPS}.
The next example, discovered in \cite{LLSvS}, is obtained as follows:
Consider the (quasi-projective) Hilbert scheme $\cH_{tc}(X)$ of twisted cubic curves $\bP^1 \simeq R \subset X$. 
When the cubic surface $\spa(R)\cap X$ is smooth, $|R|$ is a two-dimensional linear series on this cubic
surface, one of $72$ such series. The resulting dominant map to the Grassmannian $\cH_{tc}(X) \ra\Gr(4,6)$ factors
$$\cH_{tc}(X) \ra G_{\circ}(X) \ra \Gr(4,6),$$
generically a $\bP^2$-bundle followed by a degree-$72$ cover.  Combining the results of \cite{LLSvS} and \cite{AL}
with \cite[Thm.~29.2(2)]{BLMNPS}, we find that
$G_{\circ}(X)$ admits a compactification by a hyperk\"ahler manifold $G(X)$ with the following properties:
\begin{itemize}
\item{$G(X)$ has dimension eight and is deformation equivalent to the length-four Hilbert scheme of a K3 surface;}
\item{for generic $X$, $G(X)$ admits a polarization $g \in H^2(G(X),\bZ)$ with $\left(g,g\right)=2$ and $\left(g,H^2(G(X),\bZ)\right)=2\bZ$.}
\end{itemize}
Here $\left(,\right)$ denotes the Beauville-Bogomolov form.

Suppose that $X$ is special of discriminant $24$ with no additional algebraic cycles. Using the same results of \cite{BLMNPS},
we find (cf.~(\ref{BBvarpi})):
$$\Pic(G(X)) = \left<g, \varpi\right>, \quad 
\begin{array}{r|cc}
 & g & \varpi \\
 \hline
 g & 2 & 0 \\
 \varpi & 0 & -8
 \end{array}.
 $$
 The divisor class $3g-2\varpi$, satisfying $\left(3g-2\varpi,3g-2\varpi\right)=-14$, has implications
 for the birational geometry of $G(X)$. 
 \begin{prop} \label{prop:P4}
 Suppose that $X$ is a special cubic fourfold of discriminant $24$ and $G(X)$
 the associated hyperk\"ahler eightfold. If $[X] \in\cC_{24}$ is generic
 then there exists an embedding
 $$\bP^4 \hookrightarrow G(X).$$
 \end{prop}
 \begin{proof}
 This is a special case of \cite[\S 5.3]{HTray} but we sketch the key ideas.
 Recall the notation of \cite[\S 12]{BMinv}. The presence of a $\bP^4 \hookrightarrow M_v$, where $v$ is the Mukai vector,
 corresponds to algebraic classes
 $$\begin{array}{r|cc}
 	& v & a \\
\hline 
v & 6 & 3 \\
a & 3 & -2 
\end{array}
$$
in the Mukai lattice. Note that $v-2a \in v^{\perp}=H^2(M_v,\bZ)$ satisfies $\left(v-2a.,v-2a\right)=-14$ and 
$\left(v-2a,H^2(M_v,\bZ)\right)=2\bZ$. Thus $3g-2\varpi$ and $v-2a$ have the same numerical characteristics,
and consequently the same birational interpretation by the main result of \cite{BHT}.  
\end{proof}

In Section~\ref{subsect:SdPS}, we saw that each sextic del Pezzo surface has a pair $|B|,|B'|$
of two-dimensional linear series of twisted cubic curves.
\begin{ques}
How is the $\bP^4\subset G(X)$ related to the locus in $\cH_{tc}(X)$ arising from twisted cubic curves
in nodal sextic del Pezzo surfaces? 
\end{ques}

\bibliographystyle{alpha}
\bibliography{RCF24}

\end{document}